\documentclass[11pt]{amsart}
\usepackage[colorlinks, citecolor=blue, dvipdfm,
pagebackref]{hyperref}
\usepackage{graphicx}
\usepackage[all]{xy}
\usepackage{extarrows}

\usepackage{tikz}

\newtheorem{theorem}{Theorem}[section]
\newtheorem{lemma}[theorem]{Lemma}

\theoremstyle{definition}
\newtheorem{definition}[theorem]{Definition}
\newtheorem{question}[theorem]{Question}
\newtheorem{example}[theorem]{Example}

\newtheorem{proposition}[theorem]{Proposition}
\newtheorem{corollary}[theorem]{Corollary}
\newtheorem{remark}[theorem]{Remark}
\newtheorem{conjecture}[theorem]{Conjecture}

\theoremstyle{remark}

\newcommand{\be}{\begin{equation}}
\newcommand{\ee}{\end{equation}}

\numberwithin{equation}{section}

\begin{document}
\title{Chern numbers on positive vector bundles and combinatorics}
\author{Ping Li}
\address{School of Mathematical Sciences, Fudan University, Shanghai 200433, China}
\email{pinglimath@fudan.edu.cn\\
pinglimath@gmail.com}
\thanks{The author was partially supported by the National
Natural Science Foundation of China (Grant No. 12371066).}




\begin{abstract}
Combinatorial ideas are developed in this article to study Chern numbers
on ample and numerically effective vector bundles. An effective lower bound for Chern numbers of ample vector bundles is established, which makes some progress towards a long-standing question. Along this line we prove that Chern numbers on nef vector bundles
obey reverse dominance ordering, which improves upon some classical and recent results. We propose a simultaneous positivity question on
(signed) Chern numbers of compact complex or K\"{a}hler manifolds whose (co)tangent bundles are semipositive in various senses, and show that it holds true for compact homogeneous complex manifolds.
\end{abstract}

\maketitle

\section{Introduction}
In a seminal paper \cite{Gri} Griffiths introduced the now called Griffiths positivity on Hermitian (holomorphic) vector bundles. Its counterpart in algebraic geometry is ample vector bundles in the sense of Hartshorne (\cite{Har}). For line bundles these two notions are equivalent, thanks to the Kodaira embedding theorem.
It is well-known that Griffiths positivity implies ampleness, and the converse was conjectured to also hold true in \cite{Gri}, i.e., an ample vector bundle can be endowed with a Griffiths positive metric.
Griffiths semipositivity can be similarly defined. Its counterpart in algebraic geometry is the numerically effective (\emph{nef} for short) vector bundles, which was originally defined on projective manifolds and now can be extended to general compact complex manifolds (see \cite[\S1]{DPS}). Nefness is strictly weaker than Griffiths semipositivity, even for line bundles (\cite[Ex 1.7]{DPS}).

Griffiths raised in \cite{Gri} several fundamental conjectures and questions around Griffiths positivity and ampleness. One is the aforementioned equivalence between them, which is valid when the base manifold is a curve (\cite{CF}). Finding a Griffiths positive metric on a general ample vector bundle seems to be difficult, but a strategy was recently proposed by Demailly (\cite{De21}) to attack this problem.

Another question asked in \cite{Gri} is to characterize the polynomials in Chern \emph{classes} or Chern \emph{forms} which are positive as cohomology classes or differential forms (in a suitable sense) for Griffiths positive vector bundles. Griffiths introduced Griffiths positive polynomials (\cite[Definition 5.9]{Gri}) and conjectured that they are the desired polynomials (\cite[Conjecture 0.7]{Gri}). At the class level, after some partial results (\cite{BG}, \cite{Gi}, \cite{UT}), this question was completely answered by a celebrated work of Fulton and Lazarsfeld (\cite{FL}), not only for Griffiths positive vector bundles, but also for ample vector bundles. Indeed, they showed that Griffiths positive polynomials are exactly positive combinations of Schur polynomials of Chern classes (\cite[p.54]{FL}), and the latter are precisely positive polynomials for ample vector bundles (\cite[p.36]{FL}). These Fulton-Lazarsfeld type inequalities were extended to nef vector bundles over compact K\"{a}hler manifolds by Demailly, Peternell and Schneider (DPS for short in the sequel) in \cite{DPS}. At the form level, Griffiths' above question is more subtle and in its general form is largely open. For some recent progress towards it see \cite{Gu1}, \cite{Gu2}, \cite{LiMA}, \cite{DF}, \cite{Fa22-1}, \cite{Fa22-2}, \cite{Fi}, \cite{Xi} and related references therein.

An intriguing question arising from Fulton-Lazarsfeld's positivity result is what the effective or even sharp lower bounds are. Motivated by some results and conjectures of Ballico (\cite{Ba1}, \cite{Ba2}), lower bound inequalities were established by
Beltrametti, Schneider, and Sommese for \emph{very ample} vector bundles whenever the dimension of the base manifold is no more than that of the rank of the vector bundle (\cite[p.98]{BSS}). These lower bounds can fail if the vector bundle in question is merely assumed to be ample or even ample and globally generated. Such counterexamples were constructed by Hacon (see \cite{Hac} or \cite[p.123]{La}). We refer to \cite[p.121-123]{La} for a compact nice exposition on these materials. Up to now there is \emph{no} any lower bound inequality for \emph{general} ample vector bundles, to the author's best knowledge.

\emph{The main purpose} of this article is to shed new light on Chern numbers and inequalities among them for ample and nef vector bundles by incorporating some \emph{combinatorial} considerations.

The \emph{first major result} is to establish an effective lower bound for \emph{all} Chern number on \emph{general} ample vector bundles (Theorem \ref{secondmainresult}), which, as far as we are aware, is the first such lower bound. Indeed we can deduce more along this line. So our \emph{second major result} is that the Chern numbers of nef vector bundles obey the reverse dominance ordering (Theorem \ref{firstmainresult1}). As a consequence the maximal and minimal values among these Chern numbers can be directly picked up (Corollary \ref{corollary}).

Some recent evidence (\cite[Thm 7.3]{LiMA}, \cite[Thm 3.1]{LZ}) indicates that Chern numbers (resp. signed Chern numbers) of compact K\"{a}hler manifolds whose tangent (resp. cotangent) bundles are semipositive in various senses (globally generated, Griffiths semipositive, nefness) should be simultaneously positive, i.e., either they are all positive or they all vanish. In addition to proving aforementioned main results,  \emph{another purpose} in this article is to explicitly propose such phenomenon of what we call \emph{simultaneous positivity} and discuss several other important cases. In particular this holds true for compact homogeneous complex manifolds (Theorem \ref{rationalhomogeneous}). We remark that the proof of this simultaneous positivity for known cases relies, among other things, crucially on Corollary \ref{corollary}.

The rest of this article is structured as follows. The main results are stated in Section \ref{section-main results}. In Section \ref{section-preliminaries} we recall Fulton-Lazarsfeld type inequalities for ample and nef vector bundles, and basic combinatorial facts about the dominance ordering. Section \ref{section-proof of firstmainresult} is then devoted to the proof of Theorem \ref{firstmainresult1}. After discussing the lower bound $B(\lambda)$ in detail in Section \ref{section-the lower bound quantity B}, we prove Theorem \ref{secondmainresult} in Section \ref{section-Proof of secondmainresult}. The question of simultaneous positivity is proposed and discussed in Section \ref{section-simultaneous positivity} and then Theorem \ref{rationalhomogeneous} is showed in the last Section \ref{section-proofrationalhomogeneous}.

\section{Main results}\label{section-main results}
In this section some necessary combinatorial notation is introduced and main results are stated.

\subsection{Some combinatorial notation}\label{combinatorical notation}
A \emph{partition} $\lambda$ of a positive integer $n$ is a finite sequence $\lambda=(\lambda_1,\lambda_2,\ldots,\lambda_l)$ of positive integers in weakly decreasing order: $\lambda_1\geq\lambda_2\geq\cdots\geq\lambda_l>0$, such that $\sum_{i=1}^l\lambda_i=n$. These $\lambda_i$ are called \emph{parts} of $\lambda$. Denote by $l(\lambda):=l$ and call it the \emph{length} of $\lambda$. For later convenience we may regard $\lambda_{i}=0$ if $i>l(\lambda)$. This means that $\lambda$ can be identified with $(\lambda_1,\lambda_2,\ldots,\lambda_{l},0,0,\ldots)$ by adding any number of zeros. 

Let $\text{Par}(n)$ be the set of all partitions of the positive integer $n$. There are several orderings on $\text{Par}(n)$, among them the most important one is the \emph{dominance ordering} (\cite[p.7]{Ma}). Let $\lambda=(\lambda_1,\lambda_2,\ldots)$ and $\mu=(\mu_1,\mu_2,\ldots)$ belong to $\text{Par}(n)$. We call $\lambda$ \emph{dominates} $\mu$, denoted by $\lambda\geq\mu$, if
\be\label{dominance ordering}\sum_{i=1}^j\lambda_i\geq\sum_{i=1}^j\mu_i,\qquad\text{for all $j\geq1$.}\ee
In $\text{Par}(n)$ we say $\lambda>\mu$ if $\lambda\geq\mu$ and $\lambda\neq\mu$. For example we have (for simplicity the parentheses are omitted)
\be\label{par5}5>41>32>311>221>2111>11111\qquad \text{in $\text{Par}(5)$}.\ee
Also note that this dominance ordering is only \emph{partial} as soon as $n\geq6$. For instance, $3111$ and $222$ are incomparable in $\text{Par}(6)$.

With respect to this ordering the maximal and minimal values are clearly as follows.
\be\label{minimalmaximal}(n)\geq\lambda\geq(\underbrace{1,\ldots,1}_{n}),\qquad\text{for all $\lambda\in\text{Par}(n)$}.\ee

Given two positive integers $r$ and $n$, put
$$\Gamma(n,r):=\big\{\lambda=(\lambda_1,\lambda_2,\ldots)\in\text{Par}(n)~|~\lambda_1\leq r\big\}.$$
Obviously
\be\label{minimalmaximal2}
\hat{1}:=(\underbrace{r,\ldots,r}_{\lfloor\frac{n}{r}\rfloor},n-r\lfloor \frac{n}{r}\rfloor)
\geq\lambda\geq
(\underbrace{1,\ldots,1}_{n})=:\hat{0},\qquad\text{for all $\lambda\in\Gamma(n,r)$},\ee
where $\lfloor\frac{n}{r}\rfloor$ is the largest integer no more than $\frac{n}{r}$. Note also that 
$\Gamma(n,r)=\text{Par}(n)$ whenever $n\leq r$.

\subsection{Main results}
As usual we use $c_i(E)\in H^{2i}(M;\mathbb{Z})$ to denote the $i$-th Chern class of a rank $r$ complex vector bundle $E$ over some manifold $M$, and
$$c_{\lambda}(E):=\prod_{i\geq1}c_{\lambda_i}(E)\in H^{2n}(M;\mathbb{Z}),\qquad\text{if $\lambda=(\lambda_1,\lambda_2,\ldots)\in\Gamma(n,r)$}.$$
We may simply write $c_i:=c_i(E)$ and $c_{\lambda}:=c_{\lambda}(E)$ if there is no confusion from the context.

Our first major result is
\begin{theorem}\label{secondmainresult}
Let $E$ be an ample vector bundle of rank $r$ over an $n$-dimensional projective manifold $X$ with $n\leq r$. For any $\lambda\in\text{Par}(n)$, there exists a positive integer $B(\lambda)$ such that
\be\label{secondmainresultformula}\int_Xc_{\lambda}(E)\geq B(\lambda)\geq2^{l(\lambda)-1}.\ee
\end{theorem}
\begin{remark}
\begin{enumerate}
\item
By the existence of this bound $B(\lambda)$, we mean that, given $\lambda\in\text{Par}(n)$, there is a method to write down $B(\lambda)$ explicitly (see Section \ref{section-the lower bound quantity B} and Definition \ref{def of B}). 

\item
We shall see in the course of the proof that the lower bound $B(\lambda)$ we obtain is by no means optimal in general. The issue of finding the optimal lower bound appears to be difficult. Also the proof relies on the condition $r\geq n$, which is also needed in \cite{BSS} for the case of very ample vector bundles. If $r<n$, our method is invalid and it would be interesting, and probably difficult meanwhile, to find an effective lower bound in this case.
\end{enumerate}
\end{remark}

Our second major result is that Chern numbers of nef vector bundles obey the \emph{reverse} dominance ordering.
\begin{theorem}\label{firstmainresult1}
Let $E$ be a rank $r$ nef vector bundle over an $n$-dimensional compact K\"{a}hler manifold $(M,\omega)$ with $\omega$ the K\"{a}hler form, and $k$ any integer such that $1\leq k\leq n$. If $\lambda\geq\mu$ in $\Gamma(k,r)$, then
\be\label{firstmainresult1formula}
0\leq\int_Mc_{\lambda}(E)[\omega]^{n-k}\leq\int_{M}
c_{\mu}(E)[\omega]^{n-k}.\ee
\end{theorem}
\begin{remark}
If $k\leq5$, all elements in $\text{Par}(k)$ can be comparable, as illustrated in (\ref{par5}). By Theorem \ref{firstmainresult1} we can give an inequality chain for \emph{all} Chern numbers whenever $k\leq5$.
\end{remark}

Theorem \ref{firstmainresult1}, together with (\ref{minimalmaximal}) and (\ref{minimalmaximal2}), immediately yields
\begin{corollary}\label{corollary}
Let $E\longrightarrow(M,\omega)$ be a rank $r$ nef vector bundle over an $n$-dimensional compact K\"{a}hler manifold, and $1\leq k\leq n$. For any $\lambda\in\Gamma(k,r)$ we have
\be\label{consequence1.1}
0\leq\int_Mc_{k}[\omega]^{n-k}\leq\int_{M}
c_{\lambda}[\omega]^{n-k}\leq\int_{M}
c_1^k[\omega]^{n-k},\qquad\text{if $k\leq r$},
\ee
or
\be\label{consequence1.2}
0\leq\int_M(c_{r})^{\lfloor\frac{k}{r}\rfloor}
c_{k-r\lfloor\frac{k}{r}\rfloor}
[\omega]^{n-k}\leq\int_{M}
c_{\lambda}[\omega]^{n-k}\leq\int_{M}
c_1^k[\omega]^{n-k},\qquad\text{if $k>r$}.
\ee
\end{corollary}
\begin{remark}\label{remark}
\begin{enumerate}
\item
The maximal value in (\ref{consequence1.1}) and (\ref{consequence1.2}) is due to DPS (\cite[Cor.2.6, Cor.2.7]{DPS}) and plays a decisive role in establishing their structure theorem on compact K\"{a}hler manifolds with nef tangent bundles. The minimal value in (\ref{consequence1.1}), which is trivial whenever $k>r$, is due to the author and Zheng (\cite[Thm 2.9]{LZ}), which also has some related applications (see \cite[\S6-7]{LiMA}). So the minimal one in (\ref{consequence1.2}) is a complement to it.

\item
The inequality (\ref{consequence1.1}) has a similar flavor to \cite[Thm 3.2]{LiMA}, in which we require the (slightly stronger) \emph{Bott-Chern seimipositivity} on the vector bundle, and with the compensation that the base manifold can be non-K\"{a}her. See \cite[\S 3.1]{LZ} for related applications.
\end{enumerate}
\end{remark}


Apply (\ref{consequence1.1}) in Corollary \ref{corollary} and its analogue \cite[Thm 3.2]{LiMA} for Bott-Chern semipositive vector bundles, the simultaneous positivity for Chern numbers (resp. signed Chern numbers) of compact K\"{a}hler manifolds with semipositive bisectional curvature (resp. with globally generated cotangent bundle) was obtained in \cite[Thm 3.1]{LZ} (resp. in \cite[Thm 7.3]{LiMA}). We will propose this phenomenon in a more broad scope in Section \ref{section-simultaneous positivity} and discuss several other important cases therein. In particular, by refining some arguments in \cite{LiMA}, we continue to apply Corollary \ref{corollary} to verify this phenomenon for compact homogeneous complex manifolds. Namely, we have
\begin{theorem}\label{rationalhomogeneous}
Let $M$ be a compact homogeneous complex manifold. Then either all Chern numbers of $M$ are positive, in which case it is a rational homogeneous manifold; or all Chern numbers of $M$ vanish.
\end{theorem}
\begin{remark}
\begin{enumerate}
\item
A rational homogenous manifold is of the form $G/P$, where $G$ is a semi-simple complex Lie group and $P$ a parabolic subgroup. Note that $G/P$ is a Fano manifold. Conversely, a well-known conjecture du to Campana and Peternell (\cite{CP}) predicts that a Fano manifold with nef tangent bundle should be rational homogeneous.

\item
Since $G/P$ is Fano and hence it is projective. Nevertheless, in the second case in Theorem \ref{rationalhomogeneous}, the homogeneous manifold $M$ can even be \emph{non-K\"{a}hler}. For instance, there are homogeneous non-K\"{a}hler complex structures on the product of odd-dimensional
spheres $S^{2p+1}\times S^{2q+1}$ \big($(p,q)\neq(0,0)$\big) found by Calabi and Eckmann (\cite{CE}). Their Chern numbers all vanish and so they fit into the second case in Theorem \ref{rationalhomogeneous}.

\item
Theorem \ref{rationalhomogeneous} could be better appreciated after we recall/dicuss some related results and questions in Section \ref{section-simultaneous positivity}.
\end{enumerate}
\end{remark}

\section{Preliminaries}\label{section-preliminaries}
In this section Fulton-Lazarsfeld type inequalities for ample and nef vector bundles are recalled, and more combinatorics on $\text{Par}(n)$ is introduced in the framework of partially ordered sets.

\subsection{Fulton-Lazarsfeld type inequalities}
With the notation and convention set up in Section \ref{combinatorical notation} understood, we introduce the following
\begin{definition}
For each partition $\lambda=(\lambda_1,\lambda_2,\ldots,\lambda_l)\in\Gamma(n,r)$ $\big(l=l(\lambda)\big)$, the \emph{Schur polynomial} $S_{\lambda}(c_1,\ldots,c_r)\in\mathbb{Z}[c_1,\ldots,c_r]$ is defined as follows.
\be
\begin{split}
S_{\lambda}(c_1,\ldots,c_r):= \ &\text{det}(c_{\lambda_i-i+j})_{1\leq i,j\leq l}
\\
= \ &
\begin{vmatrix}\label{matrix}
c_{\lambda_1} & c_{\lambda_1+1} &\cdots &c_{\lambda_1+l-1}\\
c_{\lambda_2-1} & c_{\lambda_2} &\cdots&c_{\lambda_2+l-2}\\
\vdots&\vdots&\ddots&\vdots\\
c_{\lambda_l-l+1}&c_{\lambda_l-l+2}&\cdots&c_{\lambda_l}
\end{vmatrix}
,\end{split}\ee
where the conventions $c_0:=1$ and $c_i:=0$ if $i\notin[0,r]$ are understood.
\end{definition}
\begin{remark}
\begin{enumerate}
\item
If $\lambda_1>r$, the entries of the first row in the determinant (\ref{matrix}) all vanish and hence $S_{\lambda}(c_1,\ldots,c_r)=0$. So the restriction $\lambda\in\Gamma(n,r)$ is meaningful.

\item
If each $c_i$ is assigned weight $i$, $S_{\lambda}(c_1,\ldots,c_r)$ is then a weighted homogeneous polynomial of degree $\sum_i\lambda_i=n$. It turns out that Schur polynomials $S_{\lambda}(c_1,\ldots,c_r)$ \big($\lambda\in\Gamma(n,r)$\big) form a $\mathbb{Z}$-basis of weight homogeneous polynomials of degree $n$ in the variables $c_1,\ldots,c_r$.

\item
Schur polynomials have deep connections with combinatorics, geometry and representation theory (\cite[\S I]{Ma}, \cite{Fu}, \cite[\S 14]{Fu2}, \cite[\S 7]{Sta99}). For example, if these $c_i$ are viewed as the $i$-th complete symmetric functions of some variables, then $S_{\lambda}(c_1,\ldots,c_r)$ is exactly the \emph{Schur symmetric function} with respect to the partition $\lambda$ (\cite[(3.4)]{Ma}), which justifies the name and the $\mathbb{Z}$-basis property mentioned above (\cite[(3.2)]{Ma}). If these $c_i$ are Chern classes of the universal quotient bundles over the complex Grassmannians, $S_{\lambda}(c_1,\ldots,c_r)$ are precisely their \emph{Schubert classes} (\cite[\S14.7]{Fu2}), which is a key fact in the proof in \cite{FL}.
\end{enumerate}
\end{remark}

In our proof the Schur polynomials $S_{\lambda}(c_1,\ldots,c_r)$ with $l(\lambda)\leq2$ shall play a key role at places, and so we record them below for later convenience.
\begin{example}\label{example schur}
We have
\be\label{specialschur1}S_{(i)}(c_1,\ldots,c_r)=c_i\ee and
\be\label{specialschur2}
\begin{split}
S_{(i,j)}(c_1,\ldots,c_r)&=
\begin{vmatrix}
c_{i} & c_{i+1}\\
c_{j-1} & c_{j}
\end{vmatrix}\\
&=c_{i}c_j-c_{i+1}c_{j-1}.\qquad \big(1\leq j\leq i\leq r\big)
\end{split}\ee
\end{example}
The following fundamental result in \cite{FL} answers Griffiths' question mentioned in the Introduction at the class level completely.
\begin{theorem}[Fulton-Lazarsfeld]\label{FL inequality theorem}
Let $E$ be a rank $r$ ample vector bundle over an $n$-dimensional projective manifold $X$. Then
$$\int_{X}S_{\lambda}\big(c_1(E),\ldots,c_r(E)\big)>0$$
for any $\lambda\in\Gamma(n,r)$.
Conversely, if the integration of a degree $n$ polynomial $P(c_1,\ldots,c_r)\in\mathbb{Q}[c_1,\ldots,c_r]$ over any ample vector bundle is positive, $P$ must be a positive linear combination of some Schur polynomials.
\end{theorem}
Theorem \ref{FL inequality theorem} was extended to nef vector bundles over compact K\"{a}hler manifolds in \cite[\S 2]{DPS}.
\begin{theorem}[Demailly-Peternell-Schneider]\label{DPS inequality theorem}
Let $E$ be a rank $r$ nef vector bundle over an $n$-dimensional compact K\"{a}hler manifold $(M,\omega)$. Then
$$
\int_{M}S_{\lambda}\big(c_1(E),\ldots,c_r(E)\big)\wedge[\omega]^{n-k}\geq0$$
for any $k\in[1,n]$ and any $\lambda\in\Gamma(k,r)$.
\end{theorem}

\subsection{Combinatorics on $\text{Par}(n)$}\label{combinatorics on}
For a nice introduction to materials presented in this subsection, we refer to \cite[Chapter 3]{Sta97}.

A \emph{partially ordered set} (\emph{poset} for short) $(P,\geq)$ is a set $P$, together with a binary relation ``$\geq$" on $P$ which is reflexive ($x\geq x$, $\forall~x\in P$),  antisymmetric ($x\geq y$ and $y\geq x$ imply $x=y$), and transitive ($x\geq y$ and $y\geq z$ imply $x\geq z$). We use the obvious notation $x>y$ (resp. $x\leq y$, $x<y$) to mean $x\geq y$ and $x\neq y$ (resp. $y\geq x$, $y>x$). In this article the poset we are interested in is the set $\text{Par}(n)$ ($n\in\mathbb{Z}_{>0}$) together with the dominance partial ordering $``\geq"$ introduced in Section \ref{dominance ordering}. It is easy to check that $(\text{Par}(n),\geq)$ is indeed a poset. The dominance ordering also makes the subset $\Gamma(n,r)$ of $\text{Par}(n)$ a poset $(\Gamma(n,r),\geq)$. Various combinatorial structures on $(\text{Par}(n),\geq)$ were investigated by Brylawski in \cite{Br}, and many appearances of this ordering in mathematics were summarized in \cite[p.202-203]{Br}.

We say that the poset $P$ has a $\hat{1}$ (resp. $\hat{0}$) if there exists an element $\hat{1}\in P$ (resp. $\hat{0}\in P$) such that $x\leq\hat{1}$ (resp. $x\geq\hat{0}$) for all $x\in P$. For the poset $\Gamma(n,r)$ we have
$$\hat{1}=(\underbrace{r,\ldots,r}_{\lfloor\frac{n}{r}\rfloor},n-r\lfloor \frac{n}{r}\rfloor)\xlongequal{n\leq r}(n),\qquad\text{and}\qquad \hat{0}=(\underbrace{1,\ldots,1}_{n}),$$
as mentioned in (\ref{minimalmaximal}) and (\ref{minimalmaximal2}).

If $x,y\in P$, we say $x$ \emph{covers} $y$, denoted by $x\succ y$, if $x>y$ and if no element $z\in P$ satisfies $x>z>y$. For example in $\text{Par}(6)$ we have
\begin{eqnarray}
6\succ51\succ42\succ\left\{\begin{array}{ll}
411\\
33
\end{array}
\right.\succ321\succ
\left\{\begin{array}{ll}
3111\\
222
\end{array}
\right.\succ2211\succ21111\succ111111.
\end{eqnarray}

Given $x>y$ in $P$, a \emph{chain from $x$ to $y$}, denoted by $C(x,y)$, is a (finite) sequence
$$x=x_0\succ x_1\succ\cdots\succ x_t=y,\qquad x_i\in P,$$
where $t$ is usually called the \emph{length} of the chain $C(x,y)$. If $P$ is \emph{finite} (as $\text{Par}(n)$ and $\Gamma(n,r)$ are), a simple induction argument indicates that any $x$ and $y$ with $x>y$ can be connected by some chain $C(x,y)$, which of course may not be unique (see Example \ref{example chain1} below). So for $x>y$ in a \emph{finite} poset $P$, define
\be\label{length}l(x,y):=\max\{t~|~C(x,y):~x=x_0\succ x_1\succ\cdots\succ x_t=y\}\ee
and call it the \emph{length from $x$ to $y$}. We remark that even the chain realizing the length $l(x,y)$ may not be unique (see Example \ref{example chain2} below).

Given any $\lambda>\mu$ in $\text{Par}(n)$, a method is given in \cite{GK} for finding a chain from $\lambda$ to $\mu$  realizing the length $l(\lambda,\mu)$. Also note that if $\lambda\in\Gamma(n,r)$ and $\mu$ is dominated by it, then by definition $\mu\in\Gamma(n,r)$. So for any $\lambda\in\Gamma(n,r)$ and any chain starting from it, all the elements in the chain belong to $\Gamma(n,r)$.
\begin{example}\label{example chain1}
Take $421>2221$ in $\text{Par}(7)$. There are exactly two different chains from $421$ to $2221$ (see \cite[p.2]{GK})
\begin{eqnarray}\label{sequence}
\left\{\begin{array}{ll}
421\succ 4111\succ3211\succ 2221\\
~\\
421\succ331\succ322\succ3211\succ2221
\end{array}
\right.
\end{eqnarray}
and thus $l(421,2221)=4.$
\end{example}

\begin{example}\label{example chain2}
Consider $7>4111\in\text{Par}(7)$. There are exactly two different chains from $7$ to $4111$ realizing $l(7,4111)=5$ (see \cite[p.2]{GK}):
\begin{eqnarray}\label{sequence2}
\left\{\begin{array}{ll}
C_1(7,4111):&\qquad 7\succ61\succ52\succ43\succ421\succ4111,\\
~\\
C_2(7,4111):&\qquad 7\succ61\succ52\succ511\succ421\succ4111.
\end{array}
\right.\nonumber
\end{eqnarray}
\end{example}

\section{Proof of Theorem \ref{firstmainresult1}}\label{section-proof of firstmainresult}
When $\lambda$ covers $\mu$ in $\text{Par}(n)$ is characterized in the next proposition (see \cite[Prop. 2.3]{Br}).
\begin{proposition}\label{cover lemma}
Let $\lambda=(\lambda_1,\lambda_2,\ldots)$ and $\mu=(\mu_1,\mu_2,\ldots)$ belong to $\text{Par}(n)$. Then $\lambda\succ\mu$ if and only if there exist two indices $i<j$ such that
\begin{eqnarray}\label{cover1}
\left\{\begin{array}{ll}
\lambda_i=\mu_i+1\\
\lambda_j=\mu_j-1\\
\lambda_k=\mu_k,\qquad\text{if $k\neq i,j$}
\end{array}
\right.
\end{eqnarray}
and
\be\label{cover2}
\text{either $j=i+1$ or $\mu_i=\mu_j$.}\ee
In other words, $\lambda\succ\mu$ if and only if $\mu$ is obtained from $\lambda$ by decreasing a part of $\lambda$ by one and adding it to next nearest available part.
\end{proposition}
\begin{remark}
Take the cover relations in (\ref{sequence}) as an example. ``$4111\succ3211$" fits into the case $j=i+1$, $``3211\succ2221$" fits into the case $\mu_i=\mu_j$, and ``$421\succ331$" fits into both $j=i+1$ and $\mu_i=\mu_j$.
``$421>322$" obeys (\ref{cover1}) but does not satisfy  (\ref{cover2}) and so another partition can be inserted: $421\succ331\succ322.$
\end{remark}
\begin{proof}
Some other notions and properties in \cite{Br} are involved in the proof of \cite[Prop. 2.3]{Br}. In the sequel what we really need at places is only the direction that ``$\lambda\succ\mu$" implies (\ref{cover1}). For the reader's convenience, we provide a quite direct proof for this direction.

Assume that $\lambda\succ\mu$. Let $i$ be the first index such that $\lambda_i>\mu_i$. Namely,
\begin{eqnarray}\label{1}
\left\{\begin{array}{ll}
\lambda_k=\mu_k,\qquad \text{for $1\leq k<i$,}\\
~\\
\lambda_i>\mu_i.
\end{array}
\right.
\end{eqnarray}
And let $j>i$ be the least integer for which $\lambda_1+\cdots+\lambda_i=\mu_1+\cdots+\mu_i$. Namely,
\begin{eqnarray}\label{2}
\left\{\begin{array}{ll}
\sum_{k=1}^t\lambda_k>\sum_{k=1}^t\mu_k,\qquad \text{for $i\leq t<j$,}\\
~\\
\sum_{k=1}^j\lambda_k=\sum_{k=1}^j\mu_k.
\end{array}
\right.
\end{eqnarray}

We assert that
\be\label{3}\text{$\mu_{i-1}>\mu_i$,\qquad and\qquad $\mu_j>\mu_{j+1}.$}\ee
Indeed, by (\ref{1}) we have $\mu_{i-1}=\lambda_{i-1}\geq\lambda_i>\mu_i$. On the other hand,
\be\begin{split}\mu_j\overset{(\ref{2})}{>}\lambda_j\geq\lambda_{j+1}&=
\sum_{k=1}^{j+1}\lambda_k-\sum_{k=1}^j\lambda_k\\
&\geq\sum_{k=1}^{j+1}\mu_k-\sum_{k=1}^j\mu_k\qquad\big(\text{by
(\ref{dominance ordering}) and (\ref{2})}\big)\\
&=\mu_{j+1}
\end{split}\nonumber\ee
This completes the proof of (\ref{3}). The two strict inequalities in (\ref{3}) imply that
\be\label{4}\nu:=(\mu_1,\ldots,\mu_{i-1},\mu_i+1,\mu_{i+1},\ldots,\mu_{j-1},
\mu_{j}-1,\mu_{j+1},\ldots)\ee
is still a partition of $n$. It is immediate to see that $\lambda\geq\nu>\mu$ and hence $\lambda=\nu$, which yields the desired proof of (\ref{cover1}).
\end{proof}
\begin{remark}
Strictly speaking, in the proof above we implicitly assume in (\ref{1}) and (\ref{3}) that the index $i\geq2$. Nevertheless, even if $i=1$, the proof still works, as easily checked.
\end{remark}

A useful consequence of the assertion (\ref{cover1}) in Proposition \ref{cover lemma} is the following fact, which will be used in the proof of Theorem \ref{secondmainresult}.
\begin{corollary}\label{corollary cover}
If $\lambda\succ\mu$ in $\text{\text{Par}(n)}$ or $\Gamma(n,r)$, then either $l(\lambda)=l(\mu)$ or $l(\lambda)=l(\mu)-1$, where the latter case occurs if and only if the index $j$ appearing in (\ref{cover1}) is exactly such that  $j=l(\mu)$ and $\mu_j=1$.
\end{corollary}

A fundamental result related to Schur polynomials is that, the product of two Schur polynomials is still a linear combination of Schur polynomials, whose coefficients are \emph{nonnegative integers} determined by the remarkable Littlewood-Richardson rule (\cite[p.142]{Ma}, \cite[p.267]{Fu2}). What we need in later proof is a consequence of this rule.
\begin{proposition}\label{positivity of Schur}
Let
\be\label{positive notation}FL(k,r):=\Big\{\sum_{\lambda\in\Gamma(k,r)} a_{\lambda}S_{\lambda}(c_1,\ldots,c_r)~\Big|~\text
{all $a_{\lambda}\in\mathbb{Z}_{\geq0}$ and some $a_{\lambda}>0$}\Big\},\ee
the positive lattice generated by Schur polynomials.
Then we have
\be\label{product of schur}FL(k_1,r)\cdot FL(k_2,r)\subset FL(k_1+k_2,r)\ee
Consequently, by (\ref{specialschur1}) we have
\be\label{positive of product of Chern classes}
c_{\lambda}=\prod_{i\geq1}c_{\lambda_i}\in FL(k,r)\qquad\text{if $\lambda\in\Gamma(k,r)$}.\ee
\end{proposition}

With Propositions \ref{cover lemma} and \ref{positivity of Schur} in hand, we can now proceed to prove Theorem \ref{firstmainresult1}.

We shall show the assertion (\ref{firstmainresult1formula}) in Theorem \ref{firstmainresult1}. The nonnegativity in (\ref{firstmainresult1formula}) is well-known (see \cite[Cor.2.6]{DPS}) due to Theorem \ref{DPS inequality theorem} and (\ref{positive of product of Chern classes}):
$$\int_{M}c_{\lambda}(E)\wedge[\omega]^{n-k}\geq0,\qquad\forall\lambda\in
\Gamma(k,r).$$

In view of Theorem \ref{DPS inequality theorem} and the notation (\ref{positive notation}), Theorem \ref{firstmainresult1} follows from the next proposition.
\begin{proposition}\label{proposition assertion}
We have the following assertion.
\be\label{5}
\text{$\lambda>\mu$ in $\Gamma(k,r)$}\Longrightarrow c_{\mu}-c_{\lambda}\in FL(k,r).
\ee
Hence Theorem \ref{firstmainresult1} holds true.
\end{proposition}
\begin{proof}
We may choose a chain from $\lambda$ to $\mu$ in $\Gamma(k,r)$: $$\lambda=\lambda^{(t)}\succ\lambda^{(t-1)}\succ\cdots\succ\lambda^{(1)}
\succ\lambda^{(0)}=\mu$$
and then
\be\label{6}c_{\mu}-c_{\lambda}=\sum_{p=0}^{t-1}\big(c_{\lambda^{(p)}}-
c_{\lambda^{(p+1)}}\big).\ee
In view of (\ref{6}) as well as (\ref{positive notation}), the assertion (\ref{5}) can be reduced to show the following statement.
\be\label{7}
\text{$\lambda\succ\mu$ in $\Gamma(k,r)$}\Longrightarrow c_{\mu}-c_{\lambda}\in FL(k,r).
\ee
Indeed, if $\lambda\succ\mu$, by Proposition \ref{cover lemma} there exist indices $i<j$ such that
\be\label{8}\lambda=(\mu_1,\ldots,\mu_{i-1},\mu_i+1,\mu_{i+1},\ldots,\mu_{j-1},
\mu_{j}-1,\mu_{j+1},\ldots).\ee
Therefore,
\be\label{0}\begin{split}
c_{\mu}-c_{\lambda}&=\prod_{p\geq1}c_{\mu_p}-\prod_{p\geq1}c_{\lambda_p}\\
&=(c_{\mu_i}c_{\mu_j}-c_{\mu_i+1}c_{\mu_j-1})\prod_{p\neq i,j}c_{\mu_p}\qquad\big(\text{by (\ref{8})}\big)\\
&=S_{(\mu_i,\mu_j)}(c_1,\ldots,c_r)\prod_{p\neq i,j}S_{(\mu_p)}(c_1,\ldots,c_r)\qquad\big(\text{by (\ref{specialschur1}) and (\ref{specialschur2})}\big)\\
&\subset FL(k,r).\qquad\big(\text{by (\ref{positive notation}) and (\ref{product of schur})}\big)
\end{split}
\ee
This completes the proof of (\ref{7}) and hence Proposition \ref{proposition assertion}, and thus yields the desired proof for Theorem \ref{firstmainresult1}.
\end{proof}
\begin{remark}
Note that in the course of the proof, we don't use the full strength of Proposition \ref{cover lemma}, but only need its first assertion (\ref{cover1}), as we have mentioned in the proof of Proposition \ref{cover lemma}.
\end{remark}

\section{The lower bound $B(\lambda)$}\label{section-the lower bound quantity B}
In this section we will introduce and discuss the lower bound quantity $B(\lambda)$ appearing in Theorem \ref{secondmainresult}.

As in (\ref{minimalmaximal2}) still denote by $\hat{1}=(n)$ the maximal value in $\text{Par}(n)$ with respect to the dominance ordering.
\begin{definition}\label{def of B}
Let $B(\hat{1}):=1$.
If $\lambda\in\text{Par}(n)-\{\hat{1}\}$, we choose a \emph{longest} chain $C(\hat{1},\lambda)$ in $\text{Par}(n)$ from the maximal element $\hat{1}$ to $\lambda$ realizing $l(\hat{1},\lambda)$, the length from $\hat{1}$ to $\lambda$ \big(recall (\ref{length})\big).
\be\label{longest chain}C(\hat{1},\lambda):\qquad\hat{1}=\lambda^{(l(\hat{1},\lambda))}\succ
\lambda^{(l(\hat{1},\lambda)-1)}\succ\cdots\succ\lambda^{(1)}\succ
\lambda^{(0)}=\lambda.\ee
Then
\be\label{B}B\big(C(\hat{1},\lambda)\big):=
1+\sum_{i=0}^{l(\hat{1},\lambda)-1}
2^{l(\lambda^{(i)})-2}\ee
and
\be\label{BB}B(\lambda):=\max\big\{B\big(C(\hat{1},\lambda)\big)~\big|~
\text{longest chains $C(\hat{1},\lambda)$ from $\hat{1}$ to $\lambda$}\big\}.\ee
Here we remind the reader that $l(\lambda^{(i)})$ is the length of the partition $\lambda^{(i)}$ defined in Section \ref{combinatorical notation}.
\end{definition}
\begin{remark}
Note that in $\text{Par}(n)$ the length $l(\lambda)=1$ if and only if $\lambda=(n)$. So we always have $l(\lambda^{(i)})\geq2$ whenever $i\leq l(\hat{1},\lambda)-1$. Therefore $B\big(C(\hat{1},\lambda)\big)\in\mathbb{Z}_{>0}$ and so does $B(\lambda)$.
\end{remark}

\begin{example}\label{example}
Consider $4111\in\text{Par}(7)$. Recall from Example \ref{example chain2} that there are exactly two longest chains from $\hat{1}$ to $4111$, $C_1(\hat{1},4111)$ and $C_2(\hat{1},4111)$.
Then
\begin{eqnarray}
\left\{\begin{array}{ll}
B\big(C_1(\hat{1},4111)\big)=1+2^2+2^1+3\times2^0=10\\
~\\
B\big(C_2(\hat{1},4111)\big)=1+2^2+2\times2^1+2\times2^0=11
\end{array}
\right.\nonumber
\end{eqnarray}
and therefore $B(4111)=11$.
\end{example}

Next we present an estimate for $B(\lambda)$.
\begin{lemma}\label{estimate for B}
We have
$$
B(\lambda)\geq2^{l(\lambda)-1}$$
for all $\lambda\in\text{Par}(n)$.
\end{lemma}
\begin{remark}
In the case in Example \ref{example}, we have
$$2^{l(\lambda)-1}=
2^{4-1}=8,$$
which is less than both $B\big(C_1(\hat{1},4111)\big)=10$ and $B\big(C_2(\hat{1},4111)\big)=11$.
\end{remark}
\begin{proof}
Let $C(\hat{1},\lambda)$ be a longest chain realizing $l(\hat{1},\lambda)$ as in (\ref{longest chain}) and for simplicity let $l_0:=l(\hat{1},\lambda)$:
\be\label{l0}C(\hat{1},\lambda):\qquad\hat{1}=\lambda^{(l_0)}\succ
\lambda^{(l_0-1)}\succ\cdots\succ\lambda^{(1)}\succ
\lambda^{(0)}=\lambda.\ee
Since $\lambda^{(i+1)}\succ\lambda^{(i)}$, Corollary \ref{corollary cover} implies that
\be\label{9}
l(\lambda^{(i+1)})\leq l(\lambda^{(i)})\leq l(\lambda^{(i+1)})+1.\ee

This tells us that, with respect to the index $i$, the sequence $l(\lambda^{(i)})$ is weakly decreasing and decreases at each step by at most one. Therefore each integer between $l(\lambda^{(0)})=l(\lambda)$ and $l(\lambda^{(l_0-1)})$ appears at least once among the sequence $l(\lambda^{(i)})$:
\be\label{10}\big\{l(\lambda),l(\lambda)-1,
\ldots,l(\lambda^{(l_0-1)})\big\}
\subset\big\{l(\lambda^{(i)})~:~0\leq i\leq l_0-1\big\}.\ee
Once again by (\ref{l0}) and (\ref{9}), $l(\lambda^{(l_0-1)})\leq l(\hat{1})+1=2$ and hence we have
\be\label{11}\big\{l(\lambda),l(\lambda)-1,
\ldots,2\big\}\subset\big\{l(\lambda),l(\lambda)-1,
\ldots,l(\lambda^{(l_0-1)})\big\}.\ee

The right-hand side of (\ref{B}) can now be estimated as follows.
\be\begin{split}
B\big(C(\hat{1},\lambda)\big)&=
1+\sum_{i=0}^{l_0-1}
2^{l(\lambda^{(i)})-2}\\
&\geq1+\sum_{j=2}^{l(\lambda)}2^{j-2}
\qquad\big(\text{by (\ref{10}) and (\ref{11})}\big)\\
&=2^{l(\lambda)-1},
\end{split}\nonumber\ee
which gives the desired proof for Lemma \ref{estimate for B}.
\end{proof}

\section{Proof of Theorem \ref{secondmainresult}}\label{section-Proof of secondmainresult}
We begin to prove Theorem \ref{secondmainresult} in this section. To this end, we first introduce a notation regarding the positive lattice $FL(n,r)$ in Proposition \ref{positivity of Schur}.
\begin{definition}\label{def weight}
The \emph{weight} $\mathcal{W}(P)$ of any $P\in FL(n,r)$ \big(recall (\ref{positive notation})\big) is defined to be the sum of the (positive-integer) coefficients in front of the Schur polynomials. To be more precise, if $$P(c_1,\ldots,c_r)=\sum_{\lambda\in\Gamma(n,r)} a_{\lambda}S_{\lambda}(c_1,\ldots,c_r)\in FL(n,r),$$ then
$$
\mathcal{W}\big(P(c_1,\ldots,c_r)\big):=\sum_{\lambda\in\Gamma(n,r)}a_{\lambda}\in\mathbb{Z}_{>0}.$$
\end{definition}

In view of Theorem \ref{FL inequality theorem} and Definition \ref{def weight}, it is immediate to yield
\begin{corollary}\label{estimate lemma}
Let $E$ be a rank $r$ ample vector bundle over an $n$-dimensional projective manifold $X$. Then we have
$$\int_XP\big(c_1(E),\ldots,c_r(E)\big)\geq\mathcal{W}\big(P(c_1,\ldots,c_r)\big)$$
for any $P(c_1,\ldots,c_r)\in FL(n,r)$.
\end{corollary}

The following fact is an application of the Pieri's formula, which in turn is a special case of the general Littlewood-Richardson rule.
\begin{lemma}\label{lemma for Pieri's formula}
For any Schur polynomial $S_{\lambda}(c_1,\ldots,c_r)$ with $\lambda\in\Gamma(n,r)$, and any $1\leq i\leq r$ such that $\lambda_1+i\leq r$, we have
\be\label{weight estimate}
\mathcal{W}\big(c_i\cdot S_{\lambda}(c_1,\ldots,c_r)\big)\geq2.
\ee
\end{lemma}
\begin{proof}
Let $\lambda=(\lambda_1,\ldots,\lambda_l)$ with $l=l(\lambda)$. The Pieri's formula (\cite[p.264]{Fu2}, \cite[p.339]{Sta99}) tells us that
\be\label{12}
c_i\cdot S_{\lambda}(c_1,\ldots,c_r)=\sum_{\mu}S_{\mu}(c_1,\ldots,c_r),
\ee
where the sum is over the partitions $$\mu=(\mu_1,\ldots,\mu_l,\mu_{l+1})\in\Gamma(n+i,r)\qquad\text{\big(i.e., $l(\mu)\leq l+1$, $\mu_1\leq r$, and $\sum_{i=1}^{l+1}\mu_i=n+i$\big)}$$ such that
\be\label{13}\mu_1\geq\lambda_1\geq\mu_2\geq\lambda_2\geq\cdots
\geq\mu_{l-1}\geq\lambda_{l-1}\geq\mu_l\geq\lambda_l\geq\mu_{l+1}\geq0.\ee
Note that the two partitions $\mu=(\lambda_1+i,\lambda_2,\ldots,\lambda_l)$ and  $\mu=(\lambda_1+i-1,\lambda_2,\ldots,\lambda_l,1)$ both belong to $\Gamma(n+i,r)$, due to the assumption $\lambda_1+i\leq r$, and satisfy (\ref{13}), and hence appear on the right-hand side of (\ref{12}).
\end{proof}
\begin{remark}
The restriction condition $\lambda_1+i\leq r$ in Lemma \ref{lemma for Pieri's formula} is \emph{essential}. For instance,
$c_1c_r=S_{(r,1)}(c_1,\ldots,c_r)$ due to Example \ref{example schur} and the fact that $c_{r+1}=0$.
\end{remark}

With the aid of Lemma \ref{lemma for Pieri's formula} the crucial estimate below can be established.
\begin{proposition}\label{crucial estimate}
If $n\leq r$ and $\lambda\succ\mu$ in $\text{Par}(n)=\Gamma(n,r)$, then
$$\mathcal{W}\big(c_{\mu}(c_1,\ldots,c_r)-
c_{\lambda}(c_1,\ldots,c_r)\big)\geq2^{l(\mu)-2}.$$
\end{proposition}
\begin{proof}
First note that $\mathcal{W}(c_{\mu}-c_{\lambda})$ is well-defined as $c_{\mu}-c_{\lambda}\in FL(n,r)$ due to Proposition \ref{proposition assertion}. By Proposition \ref{cover lemma} there exist indices $i<j$ such that
\begin{eqnarray}
\left\{\begin{array}{ll}
\lambda_i=\mu_i+1\\
\lambda_j=\mu_j-1\\
\lambda_k=\mu_k,\qquad\text{if $k\neq i,j$}
\end{array}
\right.\nonumber
\end{eqnarray}
and thus \big(recall (\ref{0})\big)
\be\label{14}c_{\mu}-c_{\lambda}=(c_{\mu_i}c_{\mu_j}-c_{\mu_i+1}c_{\mu_j-1})\prod_{p\neq i,j}c_{\mu_p}\overset{(\ref{specialschur2})}{=}S_{(\mu_i,\mu_j)}(c_1,\ldots,c_r)\prod_{p\neq i,j}c_{\mu_p}.\ee
Applying Lemma \ref{lemma for Pieri's formula} repeatedly to the right-hand side of (\ref{14}) yields that its weight is no less than $2^{l(u)-2}$ and then completes its proof.
\end{proof}

We are ready to prove Theorem \ref{secondmainresult}.
\begin{proof}
The lower bound $2^{l(\lambda)-1}$ for $B(\lambda)$ in Theorem \ref{secondmainresult} has been established in Lemma \ref{estimate for B}. What we need to show now is $\int_Xc_{\lambda}(E)\geq B(\lambda)$ for any $\lambda\in\text{Par}(n)$. With Corollary \ref{estimate lemma} in hand, in order to complete the proof of Theorem \ref{secondmainresult}, it suffices to show the following inequality.
\be\label{estimate for weight}\mathcal{W}(c_{\lambda})\geq B(\lambda),\qquad\text{for any $\lambda\in\text{Par}(n)$}.\ee

We focus on the proof of (\ref{estimate for weight}). Due to (\ref{positive of product of Chern classes}) and Definition \ref{def of B} we have $\mathcal{W}(c_n)\geq1=B(\hat{1})$ and so (\ref{estimate for weight}) holds true for the maximal element $\hat{1}=(n)$. In the sequel we assume that $\lambda\in\text{Par}(n)-\{\hat{1}\}$.

Let
\be\label{longest chain2}C(\hat{1},\lambda):\qquad\hat{1}=\lambda^{(l_0)}\succ
\lambda^{(l_0-1)}\succ\cdots\succ\lambda^{(1)}\succ
\lambda^{(0)}=\lambda\ee
be any longest chain in $\text{Par}(n)$ from $\hat{1}$ to $\lambda$ realizing the length $l(\hat{1},\lambda)=l_0$ \big(recall (\ref{length})\big). Then
\be\label{15}c_{\lambda}=c_{\hat{1}}+\sum_{i=0}^{l_0-1}
(c_{\lambda^{(i)}}-c_{\lambda^{(i+1)}}).\ee
Since each summand on the right-hand side of (\ref{15}) belongs to $FL(n,r)$, still due to Proposition \ref{proposition assertion}, we have
\be\label{16}
\begin{split}
\mathcal{W}(c_{\lambda})&=\mathcal{W}(c_{\hat{1}})+
\sum_{i=0}^{l_0-1}
\mathcal{W}(c_{\lambda^{(i)}}-c_{\lambda^{(i+1)}})\\
&\geq 1+\sum_{i=0}^{l_0-1}2^{l(\lambda^{(i)})-2}\qquad
\text{(by Proposition \ref{crucial estimate})}\\
&=B\big(C(\hat{1},\lambda)\big).\qquad\text{\big(by (\ref{B})\big)}
\end{split}.\ee
Since (\ref{16}) holds true for arbitrary longest chain $C(\hat{1},\lambda)$, we have
$$\mathcal{W}(c_{\lambda})\geq\max\{B\big(C(\hat{1},\lambda)\big)\}
\overset{(\ref{BB})}{=}B(\lambda).$$
This gives the desired proof of (\ref{estimate for weight}) and thus completes the proof of Theorem \ref{secondmainresult}.
\end{proof}

\begin{remark}
After finishing the proof of Theorem \ref{secondmainresult}, it is now clear where the restriction condition $n\leq r$ plays a role. The estimate inequality in Proposition \ref{crucial estimate} needs $n\leq r$, and its proof in turn can be traced back to Lemma \ref{lemma for Pieri's formula} where we need the condition $\lambda_1+i\leq r$.
\end{remark}

\section{Simultaneous positivity}\label{section-simultaneous positivity}
\subsection{Simultaneous positivity on compact complex
manifolds with semipositive (co)tangent bundles}
There are several different but closely related semipositivity notions for (holomorphic) vector bundles over (compact) complex manifolds: global generation, Griffiths semipositivity and nefness.
In a classical work \cite{BC} Bott and Chern (implicitly) raised another semipositivity notion for Hermitian vector bundles and the author termed it \emph{Bott-Chern semipositivity} in \cite[p.24]{LiMA} and studied the Chern forms and Chern numbers on such vector bundles in detail in \cite{LiMA}.  It turns out that Bott-Chern semipositivity interpolates between global generation and Griffiths semipositivity (\cite[(4.2),(4.3)]{LiMA}). In short, we have the following sequence for these four semipositivity notions.
\be\label{four different semipositivity}\text{Global generation $\Longrightarrow$ Bott-Chern semipositivity $\Longrightarrow$ Griffiths semipositivity $\Longrightarrow$ nefness.}\ee

When the base manifold is K\"{a}hler, for each case in (\ref{four different semipositivity}) its Chern numbers are \emph{all} nonnegative (see Corollary \ref{corollary}). For the case of Bott-Chern semipositivity (and hence global generation), the base manifold can even be non-K\"{a}hler (\cite[Thm 5.1]{LiMA}).

We are interested in the case when the vector bundle in question is the (holomorphic) tangent or cotangent bundle of the base manifold. Let $M$ be a compact (connected) complex manifold with $TM$ and $T^{\ast}M$ its tangent and cotangent bundle respectively. It is well-known that $TM$ is globally generated is exactly when $M$ is homogeneous. When $M$ is K\"{a}hler, $T^{\ast}M$ is globally generated if and only if $M$ is an immersed complex submanifold of some complex torus (\cite[p.271]{Sm}). By its very definition a Hermitian metric $h$ on $M$ is such that $(TM,h)$ or $(T^{\ast}M,h)$ is Griffiths semipositive if and only if its (holomorphic) \emph{bisectional curvature} is semipositive or nonpositive respectively. Before proceeding, we introduce the following definition.
\begin{definition}
Let $M$ be an $n$-dimensional compact complex manifold whose all Chern numbers (resp. signed Chern numbers) $c_{\lambda}[M]$ \big(resp. $(-1)^nc_{\lambda}[M]$\big) are known to be nonnegative, where $\lambda\in\text{Par}(n)$. We call $M$ \emph{simultaneously positive} if these (signed) Chern numbers either are all positive or all vanish.
\end{definition}

The author observed in \cite[Thm 7.3]{LiMA}, among other things, that an immersed complex submanifold of some complex torus is simultaneously positive, i.e., a K\"{a}hler manifold with cotangent bundle globally generated is simultaneously positive. Applying the structure and uniformization theorems for compact K\"{a}hler manifolds with semipositive bisectional curvature (\cite{HSW},\cite{Wu},\cite{Mok88}), the author and Zheng observed that such manifolds also are simultaneously positive (\cite[Thm 3.1]{LZ}). Note that in each proof of these two cases Corollary \ref{corollary} plays a decisive role.

In view of these known facts, it is reasonable to propose the following question.
\begin{question}\label{question}
Let $M$ be a compact complex, Hermitian or K\"{a}hler manifold such that $TM$ or $T^{\ast}M$ belongs to some of the cases in (\ref{four different semipositivity}).
Is $M$ simultaneously positive?
\end{question}

\subsection{K\"{a}hler manifolds with nonpositive bisectional curvature or nef tangent bundles}
In this subsection we discuss Question \ref{question} for compact K\"{a}hler manifolds with nonpositive bisectional curvature or nef tangent bundle.

The consideration of Question \ref{question} for K\"{a}hler manifolds with nonpositive bisectional curvature motivates us to propose the following conjecture (\cite[\S4]{LZ}), which can be regarded as a complex analogue to the famous Hopf conjecture.
\begin{conjecture}\label{conjecture}
Let $M$ be an n-dimensional compact K\"{a}hler manifold with
nonpositive bisectional curvature whose Ricci curvature is quasi-negative.
Then its signed Euler characteristic is positive:
$(-1)^nc_n[M]>0.$
\end{conjecture}
We record some results regarding Conjecture \ref{conjecture} into the following proposition for the reader's convenience.
\begin{proposition}
Let $M$ be an $n$-dimensional compact K\"{a}hler manifold with nonpositive bisectional curvature. Then $M$ is simultaneously positive if and only if Conjecture \ref{conjecture} holds true. Conjecture \ref{conjecture} is known to be true if either $n=2$, or $n\leq4$ and with the (stronger) condition that the K\"{a}hler metric be of nonpositive (Riemannian) sectional curvature.
\end{proposition}
\begin{proof}
The equivalence between simultaneous positivity and Conjecture \ref{conjecture} has been explained in \cite[Remark 4.2]{LZ}.
The two known cases were verified in \cite[Prop.4.4]{LZ} and \cite[Cor.1.9]{LiMZ} respectively.
\end{proof}
\begin{remark}
Although there is a structure theorem for compact K\"{a}hler manifolds with nonpositive bisectional curvature due to Wu-Zheng and Liu (\cite{WZ},\cite{Liu}), on which the proof of \cite[Prop.4.4]{LZ} depends, it is not strong enough to prove Conjecture \ref{conjecture} in its full generality.
\end{remark}

Next we discuss simultaneous positivity for compact K\"{a}hler manifolds with nef tangent bundle. In \cite{DPS}, the study of K\"{a}hler manifolds with nef tangent bundle was reduced to the case of Fano manifold. The Campana-Peternell conjecture (\cite{CP}) predicts that a Fano manifold with nef tangent bundle is rational homogeneous. For such manifolds we have
\begin{proposition}
Any $n$-dimensional compact K\"{a}hler manifold with nef tangent bundle is simultaneously positive if (and only if) the Euler characteristic of any Fano manifold with nef tangent bundle is positive. In particular, this is true when $n\leq5$.
\end{proposition}
\begin{proof}
By Corollary \ref{corollary}, an $n$-dimensional compact K\"{a}hler manifold $M$ with nef tangent bundle is simultaneously positive if (and only if) $c_1^n[M]>0$ implies $c_n[M]>0$. However, for such $M$, $c^n_1[M]>0$ is equivalent to $M$ being Fano (\cite[Prop.3.10]{DPS}). So $M$ is simultaneously positive if the Campana-Peternell conjecture is true. This conjecture is currently proved when $n\leq 5$ by combining \cite{Wa} (Picard number one) with \cite{Ka} (Picard number no less than two).
\end{proof}
\begin{remark}
We refer the reader to \cite{MOSWW} for a detailed survey on the Campana-Peternell conjecture. In \cite{AMW} several conjectures related to the Betti or Hodge numbers on Fano manifolds with nef tangent bundles are proposed, which are implied by the Campana-Peternell conjecture and imply the positivity of the Euler characteristic.
\end{remark}

\section{Proof of Theorem \ref{rationalhomogeneous}}\label{section-proofrationalhomogeneous}
The proof of Theorem \ref{rationalhomogeneous} essentially is a refinement from \cite[\S6]{LiMA}, in which our main concern is the relation between projectivity and Chern numbers for compact homogeneous manifolds. For the reader's convenience, we shall still give a detailed proof here.

In order to prove Theorem \ref{rationalhomogeneous}, it suffices to show that a compact (connected) homogeneous complex manifold $M$ with some Chern number $c_{\lambda}\neq0$ must be rational homogeneous. Namely, $M$ is of the form $G/P$ with $G$ a semi-simple complex Lie group and $G$ a parabolic subgroup. The proof can be divided into the following three lemmas.

\begin{lemma}\label{lemma1}
The first Chern class $c_1(M)$ is quasi-positive.
\end{lemma}
\begin{proof}
The inequality (\ref{consequence1.1}) also holds for $TM$ when $M$ is a homogeneous complex manifold, as mentioned in Remark \ref{remark} (see \cite[Thm Cor.5.2]{LiMA}). Hence some Chern number $c_{\lambda}\neq0$ is equivalent to the Chern number $c_1^n>0$. Since $M$ is homogeneous, it can be endowed with a Hermitian metric $h$ which is Griffiths semipositive \big(see (\ref{four different semipositivity})\big), and hence the first Chern \emph{form} $c_1(h)$ with respect to $h$ is semipositive as a $(1,1)$-form. This, together with the Chern number $c_1^n>0$, implies that $c_1(h)$ is a quasi-positive closed $(1,1)$-form representing $c_1(M)$.
\end{proof}

\begin{lemma}
$M$ is projective.
\end{lemma}
\begin{proof}
Lemma \ref{lemma1} means that the anti-canonical line bundle of $M$ is quasi-positive. Then Siu and Demailly's solution to the Grauert-Riemenschneider conjecture (cf. \cite[p.96]{MM}) tells us that $M$ is Moishezon. Another result of Demailly (\cite[Cor.9.6]{De92} or \cite[Thm 4.1]{DPS}) says that a Moishezon manifold with nef tangent bundle is projective, which completes the proof of projectivity.
\end{proof}

\begin{lemma}
$M$ is a rational homogeneous manifold.
\end{lemma}
\begin{proof}
A classical result of Borel and Remmert (\cite{BR}) says that a compact (connected) homogeneous projective manifold is the product of an abelian variety and a rational homogenous manifold. Since $c_1^n>0$, there is no such factor of abelian variety in this decomposition of $M$. This completes the proof.
\end{proof}







\end{document}